\documentclass[11pt]{article}
\usepackage{smile}

\usepackage{fullpage}
\usepackage{url}
\usepackage{lscape}
\usepackage{bigints}
\usepackage{framed}
\usepackage{mdframed}
\usepackage{enumerate}
\usepackage[inline]{enumitem}
\usepackage[T1]{fontenc}
\usepackage{moresize}
\usepackage{bm}
\usepackage{bbm}
\usepackage{dsfont}
\usepackage{amsmath}
\usepackage{amssymb}
\usepackage{amsthm}
\usepackage{amsfonts}
\usepackage{stmaryrd}
\usepackage{array}
\usepackage{mathrsfs}
\usepackage{mathtools} 
\usepackage{extarrows}
\usepackage{stackrel}
\usepackage{relsize,exscale}
\usepackage{scalerel}
\usepackage{colortbl}
\usepackage[nodisplayskipstretch]{setspace}
\usepackage{color}
\usepackage[usenames,dvipsnames]{xcolor}
\usepackage{cancel}
\usepackage{soul}
\usepackage{xfrac}
\usepackage{siunitx}
\usepackage{graphicx}
\usepackage{float}
\usepackage{rotating}
\usepackage{subcaption}
\usepackage{overpic}
\usepackage[all]{xy}
\usepackage{tikz}
\usetikzlibrary{arrows,matrix,positioning,calc,automata,patterns}
\usepackage{booktabs}
\usepackage{dcolumn}
\usepackage{multirow}
\usepackage{diagbox}
\usepackage{tabularx}
\usepackage{verbatim}
\usepackage{listings}
\usepackage[ruled,vlined]{algorithm2e}
\usepackage{fancyvrb}
\usepackage{hyperref}
\usepackage[round]{natbib}
\usepackage{hyperref}
\usepackage{cleveref}
\usepackage{theoremref}
\usepackage{sectsty}

\hypersetup{
    bookmarks=true,         
    unicode=false,          
    pdftoolbar=true,        
    pdfmenubar=true,        
    pdffitwindow=false,     
    pdfstartview={FitH},    
    pdftitle={My title},    
    pdfauthor={Author},     
    pdfsubject={Subject},   
    pdfcreator={Creator},   
    pdfproducer={Producer}, 
    pdfkeywords={key1, key2}, 
    pdfnewwindow=true,      
    colorlinks=true,        
    linkcolor=blue,         
    citecolor=blue,         
    filecolor=blue,         
    urlcolor=cyan           
}

\def\scaleint#1{\vcenter{\hbox{\scaleto[3ex]{\displaystyle\int}{#1}}}}

\usepackage{stackengine}
\stackMath
\newcommand\tenq[2][1]{%
\def\useanchorwidth{T}%
\ifnum#1>1%
\stackunder[0pt]{\tenq[\numexpr#1-1\relax]{#2}}{\!\scriptscriptstyle\thicksim}%
\else%
\stackunder[1pt]{#2}{\!\scriptstyle\thicksim}%
\fi%
}

\makeatletter
\DeclareRobustCommand\widecheck[1]{{\mathpalette\@widecheck{#1}}}
\def\@widecheck#1#2{%
    \setbox\z@\hbox{\m@th$#1#2$}%
    \setbox\tw@\hbox{\m@th$#1%
       \widehat{%
          \vrule\@width\z@\@height\ht\z@
          \vrule\@height\z@\@width\wd\z@}$}%
    \dp\tw@-\ht\z@
    \@tempdima\ht\z@ \advance\@tempdima2\ht\tw@ \divide\@tempdima\thr@@
    \setbox\tw@\hbox{%
       \raise\@tempdima\hbox{\scalebox{1}[-1]{\lower\@tempdima\box
\tw@}}}%
    {\ooalign{\box\tw@ \cr \box\z@}}}
\makeatother

\DeclareMathOperator*{\las}{\overset{a.s.}{\rightarrow}}
\def\given{\,|\,}

\def\tr{\mathop{\text{tr}}\kern.2ex}

\def\P{{\mathrm P}}

\def\E{{\mathrm E}}
\def\R{{\mathbb R}}

\def\d{{\mathrm d}}

\newcommand{\pa}[1]{\left(#1\right)}
\newcommand{\bp}[1]{\left[#1\right]}
\newcommand{\ap}[1]{\left|#1\right|}
\newcommand{\cp}[1]{\left\{#1\right\}}
\newcommand{\al}[1]{\begin{align*}#1\end{align*}}

\newcommand{\pr}[1]{\mathrm{P}\pa{#1}}
\newcommand{\var}[1]{\text{Var}\pa{#1}}
\newcommand{\ep}[1]{\E\bp{#1}}
\newcommand{\cdl}{\middle |}
\newcommand{\nm}[1]{\|#1\|}
\newcommand{\bmat}[1]{\begin{bmatrix}#1\end{bmatrix}}

\newcolumntype{L}[1]{>{\raggedright\let\newline\\\arraybackslash\hspace{0pt}}m{#1}}
\newcolumntype{C}[1]{>{  \centering\let\newline\\\arraybackslash\hspace{0pt}}m{#1}}
\newcolumntype{R}[1]{>{ \raggedleft\let\newline\\\arraybackslash\hspace{0pt}}m{#1}}
\newcolumntype{d}[1]{D{.}{.}{#1}}
\newcolumntype{H}{>{\setbox0=\hbox\bgroup}c<{\egroup}@{}}
\newcolumntype{Z}{>{\setbox0=\hbox\bgroup}c<{\egroup}@{\hspace*{-\tabcolsep}}}
\newcolumntype{b}{X}
\newcolumntype{s}{>{\hsize=.5\hsize}X}

\numberwithin{equation}{section}

\newtheorem{theorem}{Theorem}[section]
\newtheorem{lemma}{Lemma}[section]

\newtheorem{corollary}{Corollary}[section]
\newtheorem{innercustomgeneric}{\customgenericname}
\providecommand{\customgenericname}{}
\newcommand{\newcustomtheorem}[2]{%
  \newenvironment{#1}[1]
  {%
   \renewcommand\customgenericname{#2}%
   \renewcommand\theinnercustomgeneric{##1}%
   \innercustomgeneric
  }
  {\endinnercustomgeneric}
}
\newcustomtheorem{customdefinition}{Definition}
\newcustomtheorem{customdefinitions}{Definitions}
\newcustomtheorem{customtheorem}{Theorem}
\newcustomtheorem{customassumption}{Assumption}
\newcustomtheorem{customlemma}{Lemma}
\newcustomtheorem{customexample}{Example}
\theoremstyle{definition}

\newtheorem{remark}{Remark}[section]

\usepackage{enumitem}
\makeatletter
\newcommand{\mylabel}[2]{#2\def\@currentlabel{#2}\label{#1}}
\makeatother

\graphicspath{{./fig3/}}

\allowdisplaybreaks

\let\check\widecheck

\begin{document}

\title{\LARGE On a rank-based Azadkia-Chatterjee correlation coefficient}

\author{Leon Tran\thanks{Department of Statistics, University of Washington, Seattle, WA 98195, USA; e-mail: {\tt leontk@uw.edu}} ~~~and~~~Fang Han\thanks{Department of Statistics, University of Washington, Seattle, WA 98195, USA; e-mail: {\tt fanghan@uw.edu}} \thanks{We thank Harry Joe for raising the question that inspired us to explore the rank-based alternative to the original Azadkia-Chatterjee correlation coefficient.}
}

\date{\today}

\maketitle

\vspace{-1em}

\begin{abstract}
Azadkia and Chatterjee \citep{azadkia2019simple} recently introduced a graph-based correlation coefficient that has garnered significant attention. The method relies on a nearest neighbor graph (NNG) constructed from the data. While appealing in many respects, NNGs typically lack the desirable property of scale invariance; that is, changing the scales of certain covariates can alter the structure of the graph. This paper addresses this limitation by employing a rank-based NNG proposed by \citet{rosenbaum2005exact} and gives necessary theoretical guarantees for the corresponding  rank-based Azadkia-Chatterjee correlation coefficient. 
\end{abstract}


\section{Introduction}

Measuring the strength of dependence between two groups of random variables, along with the associated challenges of estimation and inference, has been a central focus in statistics since its inception. Given the long and rich history of this topic, it is remarkable that Sourav Chatterjee \citep{chatterjee2021new} recently made significant progress in this area. In the bivariate case, Chatterjee introduced a novel, rank-based approach for estimating a dependence measure originally proposed by Dette, Siburg, and Stoimenov \citep{dette2013copula}. This measure satisfies R\'enyi’s criteria \citep{renyi1959measures}, taking a value of 0 if and only if the two random variables are independent, and 1 if and only if they are functionally dependent \citep{bickel2022measures}. Notably, it can also be computed in nearly linear time and has a simple normal limit that can be consistently estimated \citep{chatterjee2021new,lin2022limit,dette2024simple}.

This breakthrough has inspired a wave of research aimed at understanding the stochastic behaviors of Chatterjee’s rank correlation as well as extending it to more complex scenarios \citep{azadkia2019simple,cao2020correlations,shi2020power,gamboa2020global,deb2020kernel,huang2020kernel,auddy2021exact,lin2021boosting,fuchs2021bivariate,azadkia2021fast,griessenberger2022multivariate,zhang2022asymptotic,bickel2022measures,lin2022limit,han2022azadkia,ansari2022simple,zhang2023relationships,shi2024azadkia,strothmann2024rearranged,dette2024simple,bucher2024lack,lin2024failure,kroll2024asymptotic}. Among the most influential contributions is the work of Mona Azadkia and Sourav Chatterjee \citep{azadkia2019simple}, which extends Chatterjee’s original proposal to the multivariate case through the innovative use of a nearest neighbor graph (NNG). The resulting correlation coefficient is therefore a graph-based statistic.

However, a notable limitation of NNGs is that they generally lack scale invariance; that is, changes in the scale of some covariates can significantly alter the structure of the graph. As a result, the Azadkia-Chatterjee graph-based correlation coefficient, unlike Chatterjee’s original rank-based one \citep{chatterjee2021new}, is sensitive to scale transformations, making it challenging to interpret in practice.

To address this issue, we propose a rank-based approach to constructing NNGs. Specifically, instead of relying on Euclidean distances between the original data points, we measure proximity based on the Euclidean distance between their {\it coordinate-wise ranks}. This idea, initially introduced by \citet{rosenbaum2005exact}, has been successfully applied in a parallel line of research in causal inference \citep{rosenbaum2010design,cattaneo2023rosenbaum} and will naturally lead to a scale-invariant statistic.

Our theoretical analysis demonstrates that Azadkia-Chatterjee correlation coefficient constructed based on rank-based NNGs, which we call the rank-based Azadkia-Chatterjee correlation coefficient, provides a statistically consistent estimator of the same Dette-Siburg-Stoimenov dependence measure as the original approach \citep{azadkia2019simple}. Additionally, we derive the limiting distribution of the proposed coefficient under independence. Interestingly, when both variables are univariate, the asymptotic variance of our correlation coefficient differs from that obtained in \citet[Theorem 3.1]{shi2024azadkia}.

\section{Set up and methods}

Consider $\mX\in\R^d$ to be a $d$-dimensional random vector, $Y\in\R$ to be a random variable, and assume $\mX$ and $Y$ to be defined on the same probability space with cumulative distribution functions (CDFs) $\mF$ and $F$, respectively. This paper is interested in measuring the dependence strength between $\mX$ and $Y$. To this end, we focus on the following population quantity that was introduced by Dette, Siburg, and Stoimenov \citep{dette2013copula}:
\[
\xi=\xi(\mX,Y):=  \frac{\scaleint{4.5ex}\,\Var\big\{\E\big[\ind\big(Y\geq y\big) \given \mX \big] \big\} \d F(y)}{\scaleint{4.5ex}\,\Var\big\{\ind\big(Y\geq y\big)\big\}\d F(y)}.
\]
As shown in \citet[Proposition 1]{dette2013copula}---see, also, \citet[Theorem 1.1]{chatterjee2021new} and \citet[Theorem 2.1]{azadkia2019simple}---this measure of dependence admits the appealing properties of being 0 if and only if $Y$ is independent of $\mX$, and being 1 if and only if $Y$ is a measurable function of $\mX$. It thus answers a long-standing call of R\'enyi \citep{renyi1959measures}; see, also, \cite{schweizer1981nonparametric} and more recent reviews of this history made by Peter Bickel \citep{bickel2022measures} and Sourav Chatterjee \citep{chatterjee2022survey}.

Let $[n]:=\{1,2,\ldots,n\}$ and $\{(\mX_i, Y_i)\}_{i\in[n]}$ be $n$ independent copies of $(\mX,Y)$ with $\mX_i=(X_{i,1},\ldots,X_{i,d})^\top$. For estimating $\xi$ based only on the data points $(\mX_i, Y_i)$'s, \cite{azadkia2019simple} proposed to leverage the NNG constructed by $\mX_i$'s. More specifically, let $\tilde N(i)$ index the nearest neighbor (NN) of $\mX_i$ under the Euclidean norm $\|\cdot\|$,  let
\[
F_n(y):=\frac1n\sum_{i=1}^n\ind(Y_i\leq y)
\]
be the empirical CDF of $\{Y_i\}_{1\leq i\leq n}$, and let
\[
G_n(y):=\frac1n\sum_{i=1}^n\ind(Y_i\geq y) 
\]
be the empirical survival function. They propose to estimate $\xi$ using
\[
\xi_n^{\rm AC}:= \frac{ \sum_{i=1}^n \Big\{F_n(Y_i) \wedge F_n(Y_{\tilde N(i)}) -G_n(Y_i)^2\Big\}}{\sum_{i=1}^n G_n(Y_i)(1 - G_n(Y_i)) },
\]
where ``$\wedge$'' denotes the minimum of the two.

Unfortunately, the NNG, identified by $\{\tilde N(i);i\in[n]\}$, is not scale-invariant unless $d=1$. To remedy this, we employ the following rank-based approach that was initiated in \cite{rosenbaum2005exact}. In detail, for arbitrary $\mx=(x_1,\ldots,x_d)^\top\in\R^d$, introduce
\[
\mF_n(\mx) = (F_{n,1}(\mx), \ldots, F_{n,d}(\mx))^\top,~~{\rm with}~~F_{n,j}(\mx):=\frac1n\sum_{k=1}^n\ind(X_{k,j}\leq x_j)
\] 
being the marginal empirical CDF of $X_{1,j},\ldots,X_{n,j}$. Clearly, the equality  $\mF=\ep{\mF_n}$ holds, with $\mF=(F_1,\ldots,F_d)^\top$. Now introduce $N(i)$ as the index of the NN of $\mF_n(\mX_i)$ among $\big\{\mF_n(\mX_j); j\in[n], j\ne i\big\}$, with ties broken uniformly according to independent draws $U_i \sim {\rm Uniform}[0,1]$---the uniform distribution over $[0,1]$---that are also independent of the data. The proposed rank-based Azadkia-Chatterjee correlation coefficient is then 
\begin{align}\label{eq:xin}
\xi_n:= \frac{ \sum_{i=1}^n \Big\{F_n(Y_i) \wedge F_n(Y_{N(i)}) -G_n(Y_i)^2\Big\}}{\sum_{i=1}^n G_n(Y_i)(1 - G_n(Y_i)) }.
\end{align}
Note that the {\it only} difference between $\xi_n$ and $\xi_n^{\rm AC}$ is the replacement of the $\tilde N(i)$'s in $\xi_n^{\rm AC}$ by the rank-based NNG $N(i)$'s; the latter is naturally immune to scale alternations.

\section{Theoretical guarantees}

To justify using the proposed rank-based Azadkia-Chatterjee correlation coefficient as a safe alternative to the original $\xi_n^{\rm AC}$, we first prove that $\xi_n$ is also a consistent estimator of $\xi$ under the same condition as $\xi_n^{\rm AC}$.

\begin{theorem}\label{consistency}
As long as $Y$ is not a measurable function of $\mX$ almost surely, we have 
\[
\xi_n \text{ converges to } \xi \text{ in probability.}
\]
\end{theorem}

\begin{remark} 
\begin{enumerate}[itemsep=-.5ex, label=(\roman*)] 
\item Similar to $\xi_n^{\rm AC}$, the proposed $\xi_n$ in \eqref{eq:xin} can also be computed in $O(n \log n)$ time. This computational efficiency arises from the fact that  ranking each coordinate of $\{\mX_i; i \in [n]\}$ and $\{Y_i; i \in [n]\}$, as well as constructing an NNG of $\{\mF_n(\mX_i); i \in [n]\}$, can be accomplished in $O(n \log n)$ time. 
\item \citet[Theorem 2.2]{azadkia2019simple} demonstrated that $\xi_n^{\rm AC}$ converges to $\xi$ as long as $(Y, \mX)$ is nondegenerate—that is, $Y$ is not a measurable function of $\mX$. Theorem \ref{consistency} establishes that this also holds for $\xi_n$ under the same condition. However, we can only prove convergence in probability, not almost sure convergence as shown in \cite{azadkia2019simple}. This limitation arises due to the lack of technical tools for analyzing the tail behavior of $\xi_n$, which involves a more complex stochastic graph structure with dependent data points and ties. However, we believe that this lack of strong consistency is a technical artifact rather than an intrinsic limitation. 
\item As with $\xi_n^{\rm AC}$, the convergence of $\xi_n$ holds without requiring any additional assumptions beyond nondegeneracy. Notably, both $Y$ and $\mX$ are allowed to be discrete and include ties. 
\end{enumerate} 
\end{remark}

Our next theorem concerns the limiting distribution of $\xi_n$. For this, we follow the approach devised in \cite{shi2024azadkia} (see, also, \cite{deb2020kernel}) and focus on the limiting distribution of $\xi_n$ under independence between $Y$ and $\mX$, with additional continuity assumptions posed.

\begin{theorem}\label{an}
Assume $(\mX, Y)$ admits a Lebesgue density $f_{\mX,Y}$  and $\mF(\mX)$ admits a Lebesgue (copula) density $f$. Assume both are continuous on their supports. Assume further that $\mX$ is independent of $Y$ and $d\neq 2$. It then holds true that 
\[
\sqrt{n} \xi_n \text{ converges in distribution to } \mathcal{N}(0, \sigma_d^2), 
\]
where 
\[
\sigma_d^2 = \begin{cases}
1, & d = 1\\
\frac{2}{5} + \frac{2}{5}q_d + \frac{4}{5}o_d, & d \geq 2 \end{cases}.
\]
Here constants $q_d$ and $o_d$ are introduced in \citet[Theorem 3.1]{shi2024azadkia} as
\al{
q_d = \pa{2 - I_{3/4}\pa{\frac{d+1}{2}, \frac{1}{2}}}^{-1} ~~~{\rm and}~~~o_d = \int_{\cS_d} \exp\{- V_d(\mx_1, \mx_2)\} \, \d(\mx_1, \mx_2)
}
with
\al{
&I_{x}(a,b) = \frac{\int_0^x t^{a -1}(1-t)^{b-1} \, dt}{\int_0^1 t^{a -1}(1-t)^{b-1} \, dt}, ~~\cS_d = \cp{(\mx_1, \mx_2) \in \R^{2d} : \max\pa{\|\mx_1\|, \|\mx_2\|} < \|\mx_1 - \mx_2\|},\\
~~{\rm and}~~&V_d(\mx_1, \mx_2) = \lambda_d\Big\{\cB(\mx_1, \|\mx_1\|) \cup \cB(\mx_2, \| \mx_2\|)\Big\}.
}
Here $\cB(\mx,r)$ stands for the closed ball of radius $r$ and center $\mx$ in $(\R^d,\|\cdot\|)$ and $\lambda_d$ represents the Lebesgue measure in $\R^d$.
\end{theorem}

\begin{remark} 
Theorem 3.1 in \cite{shi2024azadkia} shows that $\sqrt{n} \xi_n^{\rm AC}$ weakly converges to a normal distribution for any fixed $d \geq 1$ under nearly the same conditions as Theorem \ref{an} above. However, Theorem \ref{an} differs in two key aspects. 
\begin{enumerate}[itemsep=-.5ex, label=(\roman*)] 
\item 
First, the limiting variance of $\xi_n$ has an interestingly distinct value from that of $\xi_n^{\rm AC}$ when $d=1$. This discrepancy arises because the convergence rates of $\mF_n$ to $\mF$ and $X_{N(i)}$ to $X_i$ differ; only when $d=1$ does the latter converge faster than the former.  
\item 
Second, Theorem \ref{an} does not cover the intriguing case of $d=2$, where the convergence rates of both $\mF_n$ to $\mF$ and $X_{N(i)}$ to $X_i$ are root-$n$, creating a singularity that current techniques cannot address. Presently, we do not have a solution for handling $d=2$. However, simulation results in Section \ref{sec:sim}, specifically Table \ref{nullvariances} ahead, suggest that $\sqrt{n}\xi_n$ also weakly converges to $\cN(0, \sigma_2^2)$ like $\sqrt{n}\xi_n^{\rm AC}$. This is further supported by theory from a related setting \citep[Theorem 1]{cattaneo2023rosenbaum}, where cases $d=1$ and $d>1$ are treated separately, with $d=2$ showing no fundamental difference from $d \geq 3$. 
\end{enumerate}
\end{remark}

\section{Simulation}\label{sec:sim}

This section presents concept-focused, small-scale simulation results that illustrate the advantages of $\xi_n$ over $\xi_n^{\rm AC}$. To this end, we use a straightforward data generation process, starting with a multivariate normal pair $(\mZ, W)$ defined as
\[
\bmat{\mZ\\ W} \sim \mathcal{N}\pa{\bm{0}, \begin{bmatrix} I_{d-1} &  \bm{0}_{d-1} & \rho\bm{1}_{d-1}\\ \bm{0}_{d-1}^T & 1 & 0 \\ \rho \bm{1}_{d-1}^T & 0 & 1 \end{bmatrix}}.
\]
Here, both $\mZ$ and $W$ are standard normal, with {\it all but the last} coordinate of $\mZ$ correlated with $W$ at a level $\rho$ that varies across simulations.

The observed data are independently generated from the pair $(\mX, Y)$, where
\[
\mX = (Z_1, \ldots, Z_{d-1}, \alpha Z_d)~~~{\rm and}~~~ Y = W.
\]
In this setup, derived from the structure of $(\mZ, W)$, only the last coordinate of $\mX$ is independent of $Y$. Increasing $\alpha$ amplifies the influence of $\mX$'s last coordinate in constructing the NNG, while leaving the structure of the rank-based NNG unchanged.

We vary the dimension $d \in \{2,3,5\}$, correlation $\rho \in \{0, 0.5, 0.9\}$, scaling factor $\alpha \in \{1,10,500\}$, and sample size $n \in \{30, 50, 100\}$. Tables \ref{table_rho0}-\ref{table_rho0.9} report the empirical means, root mean squared error (RMSE), and rejection frequencies (RF) of $\xi_n$ and $\xi_n^{\rm AC}$ for testing the null hypothesis that $\mX$ is independent of $Y$, based on 1,000 iterations. Note that the exact value of $\xi$ is not available in closed form when $\rho \neq 0$, so we estimate $\xi$ using $\xi_n^{\rm AC}$ from an independent sample of size 50,000.

Tables \ref{table_rho0}-\ref{table_rho0.9} confirm that, when $\alpha$ is large and $\rho \ne 0$, $\xi_n$ outperforms $\xi_n^{\rm AC}$ in terms of both lower RMSE and RF closer to the nominal level. This advantage becomes stronger as $\alpha$ increases. Notably, even when $\alpha=1$, $\xi_n$ consistently achieves a smaller RMSE than $\xi_n^{\rm AC}$. This observation aligns with the conventional wisdom \citep{MR79383,rosenbaum2010design} that (i) rank-based methods can be as efficient as their parametric counterparts, and (ii) nearest-neighbor-based methods can become fragile in the far tails, which rank transformation helps alleviate.

Table \ref{nullvariances} further compares the theoretical variances derived from Theorem \ref{an} with empirical variances obtained from 10,000 computed values of $\sqrt{n} \xi_n$ for $n=100$ and $d$ ranging from 1 to 10. The results are consistent with the theory and support our conjecture that the limiting variance of $\sqrt{n} \xi_n$ equals $\sigma_2^2$ when $d=2$.

\begin{table}
\centering
\caption{Means, RMSE, and rejection proportions of $\xi_n$ (``RAC'') and $\xi_n^{\rm AC}$ (``AC'') under different simulation settings and $\rho = 0$.}
\begin{tabular}{lllllllllll}\label{table_rho0}
                $d$ & $n$ & $\alpha$ & ${\rm Mean}_{\rm RAC}$ & ${\rm Mean}_{\rm AC}$ & ${\rm RMSE}_{\rm RAC}$ & ${\rm RMSE}_{\rm AC}$ & ${\rm RF}_{0.05}^{\rm RAC}$ & ${\rm RF}_{0.05}^{\rm AC}$ & ${\rm RF}_{0.1}^{\rm RAC}$ & ${\rm RF}_{0.1}^{\rm AC}$  \\\midrule
            2 & 30& 1 & -0.04 & -0.038 &
            0.188 &
            0.200&
            0.029&
            0.035&
            0.067&
            0.069\\
            &&10&&-0.038 &&0.196 &&0.029 &&0.066\\
            &&500&&-0.041 &&0.192 &&0.025 &&0.065\\
            &50& 1 & -0.028 & -0.021 &
            0.145 &
            0.146&
            0.030&
            0.036&
            0.065&
            0.074\\
            &&10&&-0.023 &&0.140 &&0.032 &&0.071\\
            &&500&&-0.022 &&0.145 &&0.031 &&0.065\\
            &100& 1 & -0.006 & -0.012 &
            0.098 &
            0.110&
            0.050&
            0.035&
            0.092&
            0.084\\
            &&10&&-0.011 &&0.109 &&0.046 &&0.085\\
            &&500&&-0.008 &&0.108 &&0.034 &&0.083\\
            3 & 30& 1 & -0.044 & -0.043 &
            0.201 &
            0.203&
            0.028&
            0.030&
            0.062&
            0.068\\
            &&10&&-0.043 &&0.201 &&0.028 &&0.062\\
            &&500&&-0.039 &&0.189 &&0.016 &&0.059\\
            &50& 1 & -0.013 & -0.031 &
            0.146 &
            0.162&
            0.031&
            0.043&
            0.075&
            0.075\\
            &&10&&-0.035 &&0.167 &&0.035 &&0.075\\
            &&500&&-0.027 &&0.149 &&0.030 &&0.066\\
            &100& 1 & -0.009 & -0.009 &
            0.104 &
            0.111&
            0.034&
            0.045&
            0.080&
            0.077\\
            &&10&&-0.010 &&0.107 &&0.036 &&0.077\\
            &&500&&-0.010 &&0.107 &&0.042 &&0.078\\
            5 & 30& 1 & -0.038 & -0.031 &
            0.211 &
            0.213&
            0.030&
            0.034&
            0.072&
            0.075\\
            &&10&&-0.029 &&0.213 &&0.040 &&0.091\\
            &&500&&-0.039 &&0.196 &&0.031 &&0.061\\
            &50& 1 & -0.030 & -0.022 &
            0.165 &
            0.167&
            0.029&
            0.034&
            0.077&
            0.079\\
            &&10&&-0.014 &&0.160 &&0.038 &&0.082\\
            &&500&&-0.020 &&0.145 &&0.025 &&0.060\\
            &100& 1 & -0.009 & -0.016 &
            0.112 &
            0.114&
            0.037&
            0.037&
            0.079&
            0.083\\
            &&10&&-0.014 &&0.121 &&0.049 &&0.098\\
            &&500&&-0.010 &&0.103 &&0.028 &&0.059\\
            \bottomrule
\end{tabular}
\end{table}

\begin{table}
\centering
\caption{Means, RMSE, and rejection proportions of $\xi_n$ (``RAC'') and $\xi_n^{\rm AC}$ (``AC'') under different simulation settings and $\rho = 0.5$.}
\begin{tabular}{lllllllllll}\label{table_rho0.5}
                $d$ & $n$ & $\alpha$ & ${\rm Mean}_{\rm RAC}$ & ${\rm Mean}_{\rm AC}$ & ${\rm RMSE}_{\rm RAC}$ & ${\rm RMSE}_{\rm AC}$ & ${\rm RF}_{0.05}^{\rm RAC}$ & ${\rm RF}_{0.05}^{\rm AC}$ & ${\rm RF}_{0.1}^{\rm RAC}$ & ${\rm RF}_{0.1}^{\rm AC}$  \\\midrule
            2 & 30& 1 & 0.104 & 0.090 &
            0.189 &
            0.196&
            0.139&
            0.107&
            0.024&
            0.196\\
            &&10&&0.029
            &&0.224 &&0.062 &&0.123\\
            &&500&&-0.042 &&0.269 &&0.021 &&0.065\\
            &50& 1 & 0.117 & 0.111 &
            0.140 &
            0.155&
            0.214&
            0.181&
            0.337&
            0.298\\
            &&10&&0.069 &&0.169 &&0.094 &&0.198\\
            &&500&&-0.018 &&0.218 &&0.028 &&0.071\\
            &100& 1 & 0.131 & 0.124 &
            0.097 &
            0.108&
            0.396&
            0.231&
            0.544&
            0.366\\
            &&10&&0.098 &&0.12 &&0.231 &&0.366\\
            &&500&&-0.007 &&0.185 &&0.037 &&0.080\\
            3 & 30& 1 & 0.233 & 0.218 &
            0.197 &
            0.215&
            0.289&
            0.306&
            0.471&
            0.459\\
            &&10&&0.105 &&0.291 &&0.123 &&0.232\\
            &&500&&-0.024 &&0.397 &&0.032 &&0.072\\
            &50& 1 & 0.264 & 0.246 &
            0.146 &
            0.162&
            0.542&
            0.484&
            0.701&
            0.626\\
            &&10&&0.137 &&0.237 &&0.215 &&0.359\\
            &&500&&-0.021 &&0.375 &&0.033 &&0.061\\
            &100& 1 & 0.286 & 0.270 &
            0.100 &
            0.114&
            0.871&
            0.823&
            0.947&
            0.905\\
            &&10&&0.194 &&0.165 &&0.560 &&0.706\\
            &&500&&-0.002 &&0.342 &&0.047 &&0.083\\
            5 & 30& 1 & 0.503 & 0.489 &
            0.454 &
            0.470&
            0.906&
            0.269&
            0.961&
            0.067\\
            &&10&&0.219 &&0.742 &&0.269 &&0.041\\
            &&500&&-0.025 &&0.983 &&0.028 &&0.067\\
            &50& 1 & 0.578 & 0.557 &
            0.372 &
            0.394&
            0.999&
            0.999&
            1.000&
            1.000\\
            &&10&&0.306 &&0.647 &&0.630 &&0.771\\
            &&500&&-0.023 &&0.974 &&0.035 &&0.064\\
            &100& 1 & 0.656 & 0.635 &
            0.288 &
            0.309&
            1.000&
            1.000&
            1.000&
            1.000\\
            &&10&&0.409 &&0.537 &&0.993 &&0.997\\
            &&500&&0.005 &&0.939 &&0.031 &&0.084\\
            \bottomrule
\end{tabular}
\end{table}

\begin{table}
\centering
\caption{Means, RMSE, and rejection proportions of $\xi_n$ (``RAC'') and $\xi_n^{\rm AC}$ (``AC'') under different simulation settings and $\rho = 0.9$.}
\begin{tabular}{lllllllllll}\label{table_rho0.9}
                $d$ & $n$ & $\alpha$ & ${\rm Mean}_{\rm RAC}$ & ${\rm Mean}_{\rm AC}$ & ${\rm RMSE}_{\rm RAC}$ & ${\rm RMSE}_{\rm AC}$ & ${\rm RF}_{0.05}^{\rm RAC}$ & ${\rm RF}_{0.05}^{\rm AC}$ & ${\rm RF}_{0.1}^{\rm RAC}$ & ${\rm RF}_{0.1}^{\rm AC}$  \\\midrule
            2 & 30& 1 & 0.512 & 0.478 &
            0.132 &
            0.161&
            0.997&
            0.903&
            0.998&
            0.958\\
            &&10&&0.227 &&0.394 &&0.310 &&0.460\\
            &&500&&-0.037 &&0.647 &&0.030 &&0.055\\
            &50& 1 & 0.536 & 0.512 &
            0.099 &
            0.117&
            1.000&
            0.995&
            1.000&
            0.999\\
            &&10&&0.334 &&0.275 &&0.779 &&0.872\\
            &&500&&-0.013 &&0.612 &&0.031 &&0.071\\
            &100& 1 & 0.563 & 0.546 &
            0.064 &
            0.076&
            1.000&
            1.000&
            1.000&
            1.000\\
            &&10&&0.425 &&0.175 &&1.000 &&1.000\\
            &&500&&0.007 &&0.584 &&0.049 &&0.101\\
            3 & 30& 1 & 0.317 & 0.299 &
            0.182 &
            0.197&
            0.503&
            0.449&
            0.680&
            0.611\\
            &&10&&0.151 &&0.314 &&0.166 &&0.275\\
            &&500&&-0.027 &&0.473 &&0.028 &&0.065\\
            &50& 1 & 0.347 & 0.332 &
            0.136 &
            0.144&
            0.766&
            0.746&
            0.887&
            0.859\\
            &&10&&0.210 &&0.239 &&0.388 &&0.538\\
            &&500&&-0.018 &&0.448 &&0.031 &&0.058\\
            &100& 1 & 0.367 & 0.354 &
            0.095 &
            0.103&
            0.985&
            0.965&
            0.993&
            0.987\\
            &&10&&0.274 &&0.164 &&0.83 &&0.906\\
            &&500&&0.003 &&0.417 &&0.036 &&0.084\\
            5 & 30& 1 & 0.118 & 0.100 &
            0.208 &
            0.232&
            0.134&
            0.143&
            0.241&
            0.244\\
            &&10&&0.060 &&0.233 &&0.064 &&0.148\\
            &&500&&-0.032 &&0.284 &&0.027 &&0.065\\
            &50& 1 & 0.130 & 0.131 &
            0.156 &
            0.163&
            0.196&
            0.206&
            0.316&
            0.322\\
            &&10&&0.089 &&0.179 &&0.135 &&0.241\\
            &&500&&-0.018 &&0.246 &&0.023 &&0.071\\
            &100& 1 & 0.151 & 0.144 &
            0.107 &
            0.115&
            0.373&
            0.352&
            0.525&
            0.504\\
            &&10&&0.115 &&0.126 &&0.275 &&0.406\\
            &&500&&-0.008 &&0.213 &&0.025 &&0.050\\
            \bottomrule
\bottomrule
\end{tabular}
\end{table}

\begin{table}
\centering
\caption{Theoretical and empirical variances of $\sqrt{n} \xi_{n}$ for $n = 100$, averaged over $10,000$ iterations.}
\begin{tabular}{lll}\label{nullvariances}
                $d$ & $\text{Theoretical Variance}$ & Empirical Variance \\\midrule
            1 &1.00 & 1.03 \\
            2 & 1.16 ({\rm conjectured}) & 1.18\\
            3 & 1.17 & 1.22\\
            4 & 1.26 & 1.26\\
            5 & 1.28 & 1.27\\
            6& 1.29 & 1.31\\
            7& 1.36 & 1.34\\
            8& 1.37& 1.35\\
            9&1.44& 1.37\\
            10&1.44& 1.43\\
            \bottomrule
\end{tabular} 
\end{table}

\section{Proofs}

For this section, we introduce additional notation. It is said that the index $j$ is equal to $N(i)$ by $i \rightarrow j$. Denote the infinity norm of a vector-valued function as $\|\mf\|_{\infty}:=\sup_{\mx}\|\mf(\mx)\|$. Let $M(i)$ denote the second nearest neighbor of $i$, and let $\hat K_n = \|\mF_n(\mX_1) - \mF_n(\mX_{M(1)})\|$. Let $\cB(\mx, r)$ denote the ball in $\R^n$ centered at $\mx \in \R^n$ with radius $r$. Also, let $\mathcal{F}_n = \sigma(\mX_1, ..., \mX_n).$ We use $\mX_n \leadsto \mX$ to denote weak convergence of $\mX_n$ to $\mX$. Furthermore, if $\mX_n \las 0$, we say that $\mX_n = o(1)$.

\subsection{Proof of consistency}

Assume that we are in the setting of Theorem \ref{consistency}. Define
\al{
Q_n := \frac{1}{n} \sum_{i=1}^n (F_n(Y_i) \wedge F_n(Y_{N(i)})) - G_n(Y_i)^2~~~{\rm and }~~~P_n := \frac{1}{n}\sum_{i=1}^n G_n(Y_i) (1 - G_n(Y_i)).
}
Also, let
$$
Q := \scaleint{4.5ex}\,\Var\big\{\E\big[\ind\big(Y\geq y\big) \given \mX \big] \big\} \d F(y)~~~{\rm and }~~~
P:= \scaleint{4.5ex}\,\Var\big\{\ind\big(Y\geq y\big)\big\}\d F(y)
$$
First, we will demonstrate that $\E[Q_n] \rightarrow Q$ and $\var{Q_n} \rightarrow 0$.

\begin{lemma}\label{cons_xni_las_xi}
As $n\rightarrow \infty$, $\mX_{N(i)} \las \mX_i$.
\end{lemma}
\begin{proof}
Without loss of generality, set $i=1$, and take any $t > 0$. Then, by the triangle inequality,
\al{
\pr{\|\mX_1 - \mX_{N(1)}\| > t} =&  \pr{ \bigcap_{j=2}^n \cp{\|\mF_n(\mX_1) - \mF_n(\mX_{j})\|> t)}}\\
\leq&  \pr{ \sup_{1 \leq k \leq n} \|\mF_n(\mX_k) - \mF(\mX_{k})\|> t/3}\\
&+ \pr{ \bigcap_{j=2}^n \{\|\mF(\mX_1) - \mF(\mX_{j})\|> t/3)\}}.
}
Using the Dvoretsky-Kiefer-Wolfowitz inequality \citep{dvoretzky1956asymptotic}, the first term tends to $0$ by letting $n$ tend to $\infty$.

Since the data are independent and identically distributed,
\al{
\pr{ \bigcap_{j=2}^n \{\|\mF(\mX_1) - \mF(\mX_{j})\|> t/3)\} \cdl \mX_1} &= \pr{\nm{\mF(\mX_1) - \mF(\mX_2)} > t/3 | \mX_1}^{n-2}\\
&=  \pr{\mF(\mX_2) \not \in \cB\pa{\mF(\mX_1), t/3} \cdl \mX_1}^{n-2}.
}
As $\mX_1$ is contained in its support with probability one, $ \pr{\mF(\mX_2) \not \in \cB\pa{\mF(\mX_1), t/3} \cdl \mX_1} < 1$ almost surely. The dominated convergence theorem demonstrates
$$
\pr{ \bigcap_{j=2}^n \{\|\mF(\mX_1) - \mF(\mX_{j})\|> t/3)\}} = \ep{\pr{\mF(\mX_2) \not \in \cB\pa{\mF(\mX_1), t/3} \cdl \mX_1}^{n-2}} \rightarrow 0
$$
as $n$ tends to $\infty$. We conclude  $\|\mX_1 - \mX_{N(1)}\| \las 0$ because $\|\mX_1 - \mX_{N(1)}\| \overset{\P}{\rightarrow} 0$ and $\|\mX_1 - \mX_{N(1)}\|$ is non-increasing in $n$.
\end{proof}

Using the same proof of Lemma 11.8 in \cite{azadkia2019simple}, we  establish
the following result; its proof is omitted due to the similarity to \cite{azadkia2019simple}.

\begin{lemma}\label{eqq}
    As $n\rightarrow \infty$, $\ep{Q_n} \rightarrow Q$. 
\end{lemma}

We establish the following lemmas to show that $\var{Q_n} \rightarrow 0$. We simply cite the first from \cite{azadkia2019simple}, which uses Lusin's Theorem to establish the convergence in probability of an arbitrary measurable function of $\mX_{N(1)}$.

\begin{lemma}\label{cons_arbmeas}
For any measurable $f : \R^d \rightarrow \R$, $f(\mX_{N(1)}) - f(\mX_1)\overset{\P}{\rightarrow} 0$ as $n$ goes to $\infty$.
\end{lemma}
The next lemma shows that $Y_{N(i)}$ given $\mX_{N(i)}$ behaves like an independent copy of $Y$ given $\mX$.
\begin{lemma}\label{cons_indgivenxni}
For distinct indices $1 \leq i,j,k \leq n$, and $\mx \in (\mathbb R^d)^3$, $(Y_i, Y_j, Y_{N(i)})$ conditional on $(\mX_i, \mX_j, \mX_{N(i)})=\mx$ is identically distributed to $(Y_i, Y_j, Y_k)$ conditional on $(\mX_i, \mX_j, \mX_k) = \mx$.
\end{lemma}
\begin{proof}
First, we demonstrate $\pr{Y_{N(i)} \leq y | \mX_1, ..., \mX_n} = \pr{Y_{N(i)} \leq y| \mX_{N(i)}}$ almost surely. Observe that 
\al{
 \pr{Y_{N(i)} \leq  y| \mX_1, ..., \mX_n}&= \sum_{j=1}^{n} \pr{Y_{N(i)} \leq  y,i \rightarrow j | \mX_1,...,\mX_n}\\
&= \sum_{j=1}^n \pr{Y_j \leq y | \mX_j} \cdot \ind(i \rightarrow j)\\
&= \sum_{j=1}^n \pr{Y_{N(i)}\leq y| \mX_{N(i)}} \cdot \ind(i\rightarrow  j)\\
&= \pr{Y_{N(i)} \leq y| \mX_{N(i)}}.
}
Next, we will show $(\mX_{N(i)}, Y_{N(i)})$ and $(\mX_1, Y_1)$ are identically distributed. For any Borel subset $B \subseteq \R^{d+1}$,
\al{
\pr{(\mX_{N(i)}, Y_{N(i)}) \in B} &= \sum_{j=1}^n \pr{(\mX_{j}, Y_{j}) \in B| i \rightarrow j}\pr{i \rightarrow j}\\
&= \frac{1}{n-1} \sum_{j \neq i} \pr{(\mX_j, Y_j) \in B| i \rightarrow j}\\
&=\pr{(\mX_1, Y_1) \in B},}
which follows by the identical distribution of $(\mX_1, Y_1),..., (\mX_n, Y_n)$.

Applying the independence of $Y_{N(i)}$ from $(\mX_i, Y_i)$ conditional on $\mX_{N(i)}$, 
\al{
\pr{Y_i \leq y_i, Y_j \leq y_j, Y_{N(i)} \leq y \cdl \mX_i, \mX_j, \mX_{N(i)}} &= \pr{Y_{N(i)} \leq y | \mX_{N(i)}} \pr{Y_i \leq y, Y_j \leq y_j\cdl \mX_i, \mX_j} 
}
almost surely. Then, using the fact that $Y_{N(i)}$ given $ X_{N(i)}$ is identically distributed to $Y_k$ given $X_k$, say, where $k \neq i,j$, we establish the claim.
\end{proof}

The next lemmas result in two asymptotic independence relationships that will greatly simplify the variance calculations in Lemma \ref{cons_var0} ahead.

\begin{lemma}\label{cons_cdl_gc}
As $n \rightarrow \infty$,
$$
\pr{\lim_{n\rightarrow \infty}\mX_{N(i)} = \mX_i\cdl \mX_{N(i)}, \mX_i, \mX_j} \overset{\P}{\rightarrow} 1\\
$$
and
$$
\pr{\lim_{n\rightarrow \infty}\|\mF_n - \mF\|_\infty = 0 \cdl \mX_{N(i)}, \mX_{N(j)}, \mX_i, \mX_j} \overset{\P}{\rightarrow} 1.
$$
\end{lemma}
\begin{proof}
By Lemma \ref{cons_arbmeas},
\al{
\pr{\lim_{n\rightarrow \infty}\mX_{N(i)} = \mX_i\cdl \mX_i \neq \mX_j, \mX_{N(i)}, \mX_i, \mX_j} &= \pr{\lim_{n\rightarrow \infty}\mX_{N(i)} = \mX_i\cdl \mX_i \neq \mX_j, \mX_i, \mX_j} + o_{\P}(1)\\
&= 1 + o_{\P}(1),
}
where the second equality comes from the fact that $\pr{\lim_{n\rightarrow \infty}\mX_{N(i)} = \mX_i\cdl \mX_1, ..., \mX_n}$ is equal to the conditional probability $ \pr{\lim_{n\rightarrow \infty}\mX_{N(i)} = \mX_i\cdl \mX_i}$; this is $1$ by the proof of Lemma \ref{cons_xni_las_xi}, and then we apply the dominated convergence theorem.

In a similar manner, by Lemma \ref{cons_arbmeas},
\al{
\pr{\lim_{n\rightarrow \infty}\|\mF_n - \mF\|_\infty = 0 \cdl \mX_{N(i)}, \mX_{N(j)}, \mX_i, \mX_j} &= \pr{\lim_{n\rightarrow \infty}\|\mF_n - \mF\|_\infty = 0 \cdl \mX_i, \mX_j} + o_{\P}(1)\\
&= 1 + o_{\P}(1),
}
where the second equality comes from the Glivenko-Cantelli theorem, completing the proof.
\end{proof}
\begin{lemma}\label{cons_Ni_distinct}
    For $i \neq j$, 
    $$
    \pr{i \rightarrow  j \cdl \mX_{N(i)}, \mX_i, \mX_j} \overset{\P}{\rightarrow} 0
    $$
    and
    $$
    \pr{i \rightarrow N(j) \cdl \mX_{N(i)}, \mX_{N(j)}, \mX_i, \mX_j} \overset{\P}{\rightarrow} 0.
    $$
\end{lemma}
\begin{proof}
Let us prove the first limit. Expanding the expression $\pr{i \rightarrow j \cdl \mX_{N(i)}, \mX_i, \mX_i}$ as 
$$
\pr{i \rightarrow j, \mX_i = \mX_j \cdl \mX_{N(i)}, \mX_i, \mX_j} + \pr{i \rightarrow j, \mX_i \neq \mX_j \cdl \mX_{N(i)}, \mX_i, \mX_j},
$$
we seek to bound the two terms of the sum. If $\pr{\mX_i = \mX_j} = 0$, then there is nothing to be shown for the first term, so suppose that $\pr{\mX_i = \mX_j} > 0$. Observe that the events, indexed by $k$, $\{\mX_k = \mX_i\}$ are mutually independent, conditional on $\mX_{N(i)}, \mX_i,$ and $ \mX_j$. Defining $A := \cp{\mX_n = \mX_i \text{ for infinitely many } n} $, the second Borel-Cantelli lemma gives
$$
\pr{A \, \cdl \mX_i, \mX_{N(i)}, \mX_j, \mX_i = \mX_j} = 1.
$$
Therefore, 
$$
\pr{i \rightarrow j\cdl  A,  \mX_i, \mX_{N(i)}, \mX_j, \mX_i = \mX_j} = \frac{1}{\ap{\{k : 1 \leq k \leq n, \mX_k = \mX_i\}}} \las 0,
$$
which demonstrates that the first term in the sum goes to $0$. 

Denote the events $G := \cp{\lim_{n\rightarrow \infty} \|\mF_n - \mF\|_\infty = 0}$ and $L := \cp{\lim_{n\rightarrow \infty} \mX_{N(i)} = \mX_i}$.
For the second term,
\al{
\pr{i \rightarrow j, \mX_i \neq \mX_j \cdl \mX_{N(i)}, \mX_i, \mX_j} &\leq \pr{i \rightarrow j \cdl \mX_i \neq \mX_j,  \mX_{N(i)}, \mX_i, \mX_j}.
}
By Lemma \ref{cons_cdl_gc}, the previous display is 
$$
\pr{i \rightarrow j, G, L\cdl \mX_i \neq \mX_j,  \mX_{N(i)}, \mX_i, \mX_j} + o_{\P}(1),
$$
which is bounded above by
$$
 \pr{i \rightarrow j, \lim_{n\rightarrow\infty} \nm{\mF_n(\mX_i) - \mF_n(\mX_{N(i)})} = 0 \cdl \mX_i \neq \mX_j,  \mX_{N(i)}, \mX_i, \mX_j} + o_{\P}(1).
$$
Since $\mX_i \neq \mX_j$, $i \not\rightarrow j$ for large enough $n$, the term above goes to $0$ in probability. 
To handle the second limit, we perform  a similar decomposition to obtain
$$
\pr{i \rightarrow N(j), \mX_i = \mX_j \cdl \mX_{N(i)}, \mX_{N(j)}, \mX_i, \mX_j} + \pr{i \rightarrow N(j), \mX_i \neq \mX_j \cdl \mX_{N(i)}, \mX_{N(j)}, \mX_i, \mX_j}.
$$
The first term is handled by observing $\mX_i = \mX_j$ implies $\mX_{N(i)} = \mX_{N(j)} = \mX_i = \mX_j$, then proceeding identically to the first limit.

Let $M := \{\lim_{n\rightarrow\infty} \mX_{N(i)} = \mX_i, \lim_{n\rightarrow\infty} \mX_{N(j)} = \mX_j\}$ so that the desired conditional probability $\pr{i \rightarrow N(j) | \mX_i \neq \mX_j, \mX_{N(i)}, \mX_{N(j)}, \mX_i, \mX_j} $ is
$$
 \pr{i \rightarrow N(j)\cdl G, M, \mX_i \neq \mX_j, \mX_{N(i)}, \mX_{N(j)}, \mX_i, \mX_j} + o_{\P}(1).
$$

Proceeding analogously, an upper bound of the previous display is
$$
\pr{i \rightarrow N(j), \lim_{n\rightarrow\infty}\nm{\mF_n(\mX_i) - \mF_n(\mX_{N(i)})} = 0 \, | \mX_i \neq \mX_j, \mX_{N(i)}, \mX_{N(j)}, \mX_i, \mX_j} + o_{\P}(1).
$$

By Lemma \ref{cons_cdl_gc} and for an $\epsilon >0$ small enough,
$$
\pr{\nm{\mF_n(\mX_i) - \mF_n(\mX_{N(j)})} > \epsilon\, \cdl \mX_i \neq \mX_j, \mX_{N(i)}, \mX_{N(j)}, \mX_i, \mX_j} \overset{\P}{\rightarrow} 1,
$$
from which we conclude that the original term is $o_{\P}(1)$.
\end{proof}
\begin{lemma}\label{cons_plim_yni}
If $i \neq j$, 
\al{
\pr{Y_j \leq Y_i \wedge Y_{N(i)}|\mX_{N(i)}, \mX_i, \mX_j} &\overset{\P}{\rightarrow} \pr{Y_j \leq Y_i \wedge Y_i'| \mX_i, \mX_j};
}
If $i \neq j \neq k \neq l$,
\al{
\pr{Y_k \leq Y_i \wedge Y_{N(i)}, Y_l \leq  Y_j \wedge Y_{N(j)}|\mX_{N(i)}, \mX_{N(j)}, \mX_i, \mX_j, \mX_k, \mX_l}&\overset{\P}{\rightarrow} \pr{Y_k \leq Y_i \wedge Y'_i| \mX_i, \mX_k}\\
&\,\cdot\pr{Y_l \leq Y_j \wedge Y'_j| \mX_j, \mX_l},
}
where $Y_i'$ is independently sampled from $Y| \mX = \mX_i'$.
\end{lemma}
\begin{proof}

By Lemma \ref{cons_indgivenxni},
$$
\pr{Y_j \leq Y_i \wedge Y_{N(i)} \cdl \mX_{N(i)} = \mx, \mX_i = \mx_i, \mX_j = \mx_j} = \pr{Y_j \leq Y_i \wedge Y_i' \cdl \mX_i' = \mx, \mX_i = \mx_i, \mX_j = \mx_j}
$$
where $(\mX_i', Y_i')$ is an independent and identically distributed copy of the data. Then, the function $f : (\mx, \mx_i, \mx_j) \mapsto \pr{Y_j \leq Y_i \wedge Y_i' \cdl \mX_i' = \mx, \mX_i = \mx_i, \mX_j=\mx_j}$ is measurable by the existence of the regular conditional probability. Applying Lemma \ref{cons_arbmeas} yields $f(\mX_{N(i)}, \mX_i, \mX_j) \overset{\P}{\rightarrow} f(\mX_i, \mX_i, \mX_j)$, from which we conclude the first limit.

We proceed in an analogous way for the remaining limit. First, we define the functions 
\al{
g&: (\mx_i', \mx_j', \mx_i, \mx_j, \mx_k, \mx_l) \mapsto\\ 
&\pr{Y_k \leq Y_i \wedge Y_{N(i)}, Y_l \leq  Y_j \wedge Y_{N(j)}\cdl \mX_{N(i)} = \mx_i', \mX_{N(j)} = \mx_j', \mX_i = \mx_i, \mX_j = \mx_j, \mX_k = \mx_k, \mX_l = \mx_l}.
}

Then, by Lemmas \ref{cons_indgivenxni} and \ref{cons_Ni_distinct},
\al{
g& (\mx_i', \mx_j', \mx_i, \mx_j, \mx_k, \mx_l) = \\ 
&\pr{Y_k \leq Y_i \wedge Y_i', Y_l \leq  Y_j \wedge Y_j' \cdl \mX_{i}' = \mx_i', \mX_{j}' = \mx_j', \mX_i = \mx_i, \mX_j = \mx_j, \mX_k = \mx_k, \mX_l = \mx_l} + o_{\P}(1),
}
where $(Y_i', Y_j') | (\mX_i', \mX_j') = \mx$ has the same distribution as $(Y_{N(i)}, Y_{N(j)}) | (\mX_{N(i)}, \mX_{N(j)}) = \mx$. Applying Lemma \ref{cons_arbmeas}, 
\al{
g(\mX_{N(i)}, \mX_{N(j)}, \mX_i, \mX_j, \mX_k, \mX_l) &\overset{\P}{\rightarrow} g(\mX_i, \mX_j, \mX_i, \mX_j, \mX_k, \mX_l). 
}
Using the independence of the data, the second limit is proven, concluding the proof.
\end{proof}

\begin{lemma}\label{cons_var0}
As $n\rightarrow \infty$, $\var{Q_n} \rightarrow 0$.
\end{lemma}

\begin{proof}

First, we will show $\ep{\var{Q_n|\mX_1, ..., \mX_n}} \rightarrow 0$.
Expanding the variance of the sum,
\al{
\var{Q_n|\mX_1, ..., \mX_n} &= \frac{1}{n^2} \var{\sum_{i=1}^n F_n(Y_i) \wedge F_n(Y_{N(i)})\cdl \mathcal{F}_n}\\
&= \frac{1}{n^2}\sum_{i=1}^n\sum_{j=1}^n \text{Cov}\pa{F_n(Y_i) \wedge F_n(Y_{N(i)}), F_n(Y_j) \wedge F_n(Y_{N(j)}) \cdl \mathcal{F}_n}\\
&= \frac{1}{n^4}\sum_{i=1}^n\sum_{j=1}^n \text{Cov}\pa{\sum_{k=1}^n \ind(Y_k \leq Y_i \wedge Y_{N(i)}), \sum_{l=1}^n \ind(Y_l \leq Y_j \wedge Y_{N(j)}) \cdl \mathcal{F}_n}\\
&= \frac{1}{n^4} \sum_{i,j,k,l=1}^n \text{Cov}\pa{\ind(Y_k \leq Y_{i} \wedge Y_{N(i)}), \ind(Y_l \leq Y_{j} \wedge Y_{N(j)})\cdl \mathcal{F}_n}.
}
Lemmas \ref{cons_indgivenxni} and Lemma \ref{cons_plim_yni} yield the equality, for $i
\neq j\neq k \neq l$,
\al{
\pr{Y_k \leq Y_{i} \wedge Y_{N(i)}, Y_l \leq Y_{j} \wedge Y_{N(j)}\cdl \mathcal{F}_n} \overset{p}{\rightarrow } \pr{Y_k \leq Y_{i} \wedge Y_{N(i)} \cdl \mX_i, \mX_k}\pr{Y_l \leq Y_{j} \wedge Y_{N(j)} \cdl \mX_j, \mX_l}.
}
Equivalently, $\text{Cov}\pa{A_{i,k}, A_{j,l} \cdl \mathcal{F}_n} = o_{\P}(1).$
As a consequence, there are only $O(n^3)$ indices for which 
$$
\text{Cov}\pa{\ind(Y_k \leq Y_{i} \wedge Y_{N(i)}), \ind(Y_l \leq Y_{j} \wedge Y_{N(j)})\cdl \mathcal{F}_n} = o_{\P}(1).
$$
Additionally, since this $o_{\P}(1)$ term is bounded,
$$
\frac{1}{n^4} \sum_{i,j,k,l=1}^n \text{Cov}\pa{\ind(Y_k \leq Y_{i} \wedge Y_{N(i)}), \ind(Y_l \leq Y_{j} \wedge Y_{N(j)})\cdl \mathcal{F}_n} = o_{\P}(1)
$$
and is also bounded. Applying the dominated convergence theorem yields 
\[
\ep{\var{Q_n|\mX_1, ..., \mX_n}} \rightarrow 0.
\]

Next, we will demonstrate $\var{\ep{Q_n |\mX_1, ..., \mX_n}} \rightarrow 0.$ Using the first limit of Lemma \ref{cons_plim_yni},
\al{
\ep{Q_n|\mX_1, ..., \mX_n} &= \frac{1}{n^2}\sum_{i=1}^n \sum_{j=1}^n  \ep{\ind(Y_j \leq Y_{i} \wedge Y_{N(i)})|\mX_1, ..., \mX_n}\\
&= \frac{1}{n^2}\sum_{i=1}^n \sum_{j:j\neq i}  \ep{\ind(Y_j \leq Y_{i} \wedge \tilde Y_{i,j})|\mX_i, \mX_j} + o_{\P}(1)
}
by Lemma \ref{cons_plim_yni}, where $\tilde Y_{i,1}, ..., \tilde Y_{i, n}$ denote independent copies of $Y_i$. It suffices to show the variance of the above display tends to $0$. Letting $Z_{i,j} = \ep{\ind(Y_j \leq Y_{i} \wedge \tilde Y_{i,j}|\mX_i, \mX_j}$,
$$
\var{\ep{Q_n|\mX_1, ..., \mX_n}} = \frac{1}{n^4} \sum_{i,j,k,l: i\neq j, k\neq l } \text{Cov}\pa{Z_{i,j}, Z_{k,l}} + o_{\P}(1) .
$$
Since $Z_{i,j}$ and $Z_{k,l}$ are independent if $i \neq j \neq k\neq l$, there are at most $O(n^3)$ indices for which $\text{Cov}\pa{Z_{i,j}, Z_{k,l}} \neq 0$, from which the conclusion follows.
\end{proof}

For completeness, we state a claim proven in \cite{azadkia2019simple} which calculates the limit of $P_n$. 

\begin{lemma}\label{cons_pnlim}
As $n \rightarrow \infty$, $P_n \las P$.
\end{lemma}
Combining all our previous results, we arrive at the proof of Theorem \ref{consistency} below.

\begin{proof}[Proof of Theorem \ref{consistency}] 
Notice that $\xi_n = Q_n/P_n$. Consider the intermediate value 
$$
\tilde \xi_n := \frac{Q_n}{\int \var{\ind(Y \geq y)} \, dF(y)}.
$$
By the continuous mapping theorem, $\xi_{n} - \tilde \xi_n \las 0$, so it suffices to show $\tilde \xi_n \overset{\P}{\rightarrow} \xi$. By Lemma \ref{eqq}, $\ep{\xi_n} \rightarrow \xi$, and by Lemma \ref{cons_var0}, $\var{\tilde \xi_n} \rightarrow 0$. We conclude by Chebyshev's inequality that $\xi_n \overset{\P}{\rightarrow} \xi$.
\end{proof}

\subsection{Results on nearest-neighbor graphs with dependent nodes}

In this section, we assume the conditions of Theorem \ref{an}. We consider the directed nearest neighbor graph $\mathcal{G}_n$ associated with the marginal rank vectors of the data, where $i \rightarrow j$ if $N(i) = j$, and $i \leftrightarrow j$ if $N(i) = j$ and $N(j) = i$. Let $T_{i,j} := \ind(i \leftrightarrow j)$ and $C_{i,j,k} := \ind(i \rightarrow k, j \rightarrow k)$, where $1\leq i,j,k \leq n$. This section will develop results necessary to calculate the limiting values of 
$$
\ep{\frac{1}{n}\sum_{i,j,k=1}^n C_{i,j,k}} \text{ and }  
\ep{\frac{1}{n}\sum_{i,j=1}^n T_{i,j}}
$$
as $n\rightarrow \infty$.

First, we may evaluate the limits when $d = 1$.
\begin{lemma}\label{an_calc_d1}
Assume $d = 1$. As $n\rightarrow \infty$,
\[
\ep{\frac{1}{n}\sum_{i,j,k=1}^n C_{i,j,k}} \rightarrow \frac{1}{2}~~~{\rm and}~~~
\ep{\frac{1}{n}\sum_{i,j=1}^n T_{i,j}} \rightarrow \frac{1}{2}
\]
\end{lemma}
\begin{proof}

Say $n > 3$. Since $d = 1$, we have $C_{i,j,k} = 0$ unless $1 \leq i,j,k \leq n$ are in consecutive order. Assuming that $i < k < j$ see that $\pr{i \rightarrow k, j\rightarrow k} = \frac{1}{4}$ if $i\neq 1$ and $ j\neq n$ and $\pr{i \rightarrow k, j \rightarrow k} = \frac{1}{2}$ if either $i = 1$ or $j = n$. Since we may interchange $i$ and $j$, 
$$
\ep{\frac{1}{n} \sum_{i,j,k = 1}^n C_{i,j,k}} = \frac{2}{n}\pa{1 + \frac{n-4}{4}} \rightarrow \frac{1}{2}.
$$

Furthermore, $\pr{i \leftrightarrow j} = \frac{1}{2}$ if either $i = 1$ or $j=n$, and is $\frac{1}{4}$ otherwise. A similar calculation reveals that 
$$
\ep{\frac{1}{n}\sum_{i,j=1}^n T_{i,j}} = \frac{2}{n}\pa{1 + \frac{n-2}{4}} \rightarrow \frac{1}{2},
$$
concluding the proof.
\end{proof}
For the remainder of this section, we assume that $d > 2$.  For the following proofs, define the ball $\cB_n := \cB(\mF(\mX_1), n^{-1/d})$, its slight enlargement $\cB_n^\delta := \cB(\mF(\mX_1), n^{-1/d} + \delta)$, and the corresponding products $\cW_n := \cB_n \times (\R^d \setminus \cB_n)^{n-1}$ and $\cW_n^\delta := \cB_n^{\delta} \times (\R^d \setminus \cB_n^{\delta})^{n-1}$. Analogously, define the re-centered ball $\hat \cB_n := \cB(\mF_n(\mX_1), n^{-1/d})$ and $\hat \cW_n := \hat \cB_n \times (\R^d \setminus \hat \cB_n)^{n-1}$.

\begin{lemma}\label{an_depequiv}
 As $n \rightarrow \infty$, 
\[
\frac{\pr{(\mF_n(\mX_1), ..., \mF_n(\mX_n)) \in \hat \cW_n\cdl \mX_1}}{\pr{(\mF(\mX_1), ..., \mF(\mX_n))\in  \cW_n\cdl \mX_1}} \las 1.
\]
\end{lemma}
\begin{proof}
Set $\delta_n = n^{-\frac{1}{2} + \epsilon}$ where $0 < \epsilon < \frac{1}{2} - \frac{1}{d}$. Let $\mF_n^{-}(\mx) := \pa{\frac{1}{n-1} \sum_{k=2}^n \ind( x_{j} \geq X_{k,j})}_{1 \leq j \leq d}$.
The triangle inequality and a union bound give
\al{
\pr{\|\mF_n - \mF\|_\infty > 2\delta_n\cdl \mX_1} &\leq  \pr{\|\mF_n - \mF_n^-\|_\infty > \delta_n\cdl \mX_1} + \pr{\|\mF_n^- - \mF\|_\infty > \delta_n}.
}
By definition, $\|\mF_n - \mF_n^-\|_\infty = O(n^{-1})$, so the first term goes to $0$ as $n$ tends to $\infty$. 

By the Dvoretsky-Kiefer-Wolfowitz inequality, \citep{dvoretzky1956asymptotic} and our choice of $\delta_n$, the second term tends to $0$ as $n$ approaches infinity.

Next, the triangle inequality gives
\al{
\pr{(\mF_n(\mX_1), ..., \mF_n(\mX_n)) \in \hat \cW_n, \|\mF_n - \mF\|_\infty \leq \delta_n\cdl \mX_1} &\leq \pr{(\mF(\mX_1), ..., \mF(\mX_n))\in \cW_n^{-2\delta_n}|\mX_1}\\
&=  \pr{\mF(\mX_2) \not \in  \cB_n^{-2\delta_n}\cdl \mX_1}^{n-1}.
}
Notice
$$
\frac{\pr{\mF(\mX_2) \not\in \cB_n^{-2\delta_n}\cdl \mX_1}}{\pr{\mF(\mX_2) \not\in \cB_n\cdl \mX_1}} = \pa{1 + \frac{\pr{\mF(\mX_2) \in \cB_n \setminus \cB_n^{-2\delta_n}\cdl \mX_1}}{1 - \pr{\mF(\mX_2) \in \cB_n\cdl \mX_1}}},
$$
and $1 - \pr{\mF(\mX_2) \in \cB_n\cdl \mX_1}$ tends to $1$ as $n$ approaches $\infty$ because $\cB_n$ decreases to the set $\{\mX_1\}$ as $n \rightarrow \infty$, and the probability $\mF(\mX_2)$ lies in this set is $0$, by the continuity of $\mF(\mX)$. A further calculation reveals that
$$
\pr{\mF(\mX_2) \in \cB(\mF(\mX_1), n^{-1/d})\setminus \cB_n^{-2\delta_n }\cdl \mX_1} = O(n^{-1} - (n^{-1/d} - 2\delta_n)^d)= o\pa{n^{-1}}
$$
using the continuity of $f$. From this,
$$
\pa{\frac{\pr{\mF(\mX_2) \not\in \cB_n^{-2\delta_n}\cdl \mX_1}}{\pr{\mF(\mX_2) \not\in \cB_n\cdl \mX_1})}}^{n-1} \las 1
$$
so that
$$
\limsup_{n\rightarrow \infty}\frac{\pr{(\mF_n(\mX_1), ..., \mF_n(\mX_n)) \in \hat \cW_n\cdl \mX_1}}{\pr{(\mF(\mX_1), ..., \mF(\mX_n))\in  \cW_n\cdl \mX_1}} \leq  1.
$$

Analogously,
$$
\pa{\frac{\pr{\mF(\mX_2) \not\in \cB_n^{2\delta_n}\cdl \mX_1}}{\pr{\mF(\mX_2) \not\in \cB_n\cdl \mX_1})}}^{n-1} \las 1
$$
so that
$$
\liminf_{n\rightarrow \infty}\frac{\pr{(\mF_n(\mX_1), ..., \mF_n(\mX_n)) \in \hat \cW_n\cdl \mX_1}}{\pr{(\mF(\mX_1), ..., \mF(\mX_n))\in  \cW_n\cdl \mX_1}} \geq  1.
$$
Putting these two bounds together, we obtain 
$$
\frac{\pr{(\mF_n(\mX_1), ..., \mF_n(\mX_n)) \in \hat \cW_n\cdl \mX_1}}{\pr{(\mF(\mX_1), ..., \mF(\mX_n))\in  \cW_n\cdl \mX_1}} \las 1
$$
as desired.
\end{proof}

Next, we define the nearest-neighbor distances

$$
\hat D_n := \nm{\mF_n(\mX_1) - \mF_n(\mX_{N(1)})}~~~{\rm and}~~~
D_n := \nm{\mF(\mX_1) - \mF(\mX_{\check N(1)})},
$$
where $\check N(1)$ is the index of the nearest neighbor according to the Euclidean metric, to $\mF(\mX_1)$, among $\mF(\mX_2), ..., \mF(\mX_n)$, with ties broken uniformly at random. 

The main consequence of Lemma \ref{an_depequiv} is that the two distances possess the same limiting law when scaled by $\sqrt{n}$.
\begin{corollary}\label{an_hatdn_tight}
     Conditional on $\mX_1$, $n\hat D_n^d$ and $n D_n^d$ have the same limiting exponential law.
\end{corollary}
\begin{proof}
    First, fixing $v > 0$ and letting $r_n := v/n$,
    \al{
    \pr{n\hat D_n^d > v| \mX_1} &= \pr{\min_{2 \leq i \leq n} \nm{\mF_n(\mX_i) - \mF_n(\mX_1)} > r_n | \mX_1}.
    }
    Applying Lemma \ref{an_depequiv}, we calculate
    $$
    \frac{\pr{\min_{2 \leq i \leq n} \nm{\mF_n(\mX_i) - \mF_n(\mX_1)} > r_n | \mX_1}}{\pr{\min_{2 \leq i \leq n} \nm{\mF(\mX_i) - \mF(\mX_1)} > r_n | \mX_1}} \las 1,
    $$
    so it suffices to compute the limit of the denominator. Using the independence and identical distribution of the $\mF(\mX_i)$, 
    \al{
    \pr{\min_{2 \leq i \leq n} \nm{\mF(\mX_i) - \mF(\mX_1)} > r_n \cdl \mX_1}
    &= \pa{1 - \pr{\nm{\mF(\mX_2) - \mF(\mX_1)} \leq r_n | \mX_1}}^{n-1}.
    }
    Since we have already assumed $\mF$ to admit a continuous density, 
    \al{
    \pr{\nm{\mF(\mX_2) - \mF(\mX_1)} \leq r_n | \mX_1} &= \int_{\cB\pa{\mF(\mX_1), r_n}} f(z) \, dz\\
    &=  \int_{\cB\pa{\mF(\mX_1), r_n}} f\pa{\mF(\mX_1)} + o(1) \, dz\\
    &= f(\mF(\mX_1)) \cdot \lambda_d(\mathbf{B}(\mF(\mX_1), r_n)) + \int_{\mathbf{B}(\mF(\mX_1), r_n)}o(1) \, dz.
    }
     Using the continuity of $f$, along with the translation-invariance and scaling properties of $\lambda_d$,
    \al{
    \pa{1 - \lambda_d\pa{\cB\pa{\mF(\mX_1), f(\mF(\mX_1))^{1/d} r_n}} + o(n^{-1})}^{n-1}
    &= \pa{1 - \frac{v f(\mF(\mX_1))}{n}\lambda_d\pa{\cB(0, 1)} + o(n^{-1})}^{n-1}\\
    &\rightarrow \exp\pa{-vf(\mF(\mX_1)) \cdot\lambda_d\pa{\cB(0,1)}}.
    }
    Thus, conditional on $\mX_1$, $n \hat D_n^d$ and $n \hat D_n^d$ both converge weakly to an exponential random variable with parameter $f(\mF(\mX_1)) \cdot \lambda_d(\cB(0,1))$.
\end{proof}

Corollary $\ref{an_hatdn_tight}$ states that the nearest neighbor to the first data point, among weakly dependent data points, converges at a slower than root-$n$ rate for dimension larger than $2$. As a result, the next two lemmas demonstrate that computing nearest neighbors based on empirical CDFs or population CDFs are equivalent, in the limit. 

\begin{lemma}\label{an_kn_dn}
 For any sequence $\delta_n$ satisfying $\|\mF_n - \mF\|_\infty = o_p(\delta_n)$, 
$$
\pr{\hat K_n < \hat D_n + \delta_n} \rightarrow 0.
$$
\end{lemma}
\begin{proof}
     Expanding this probability, and noting that the $\mF_n(\mX_i)$ are identically distributed,
    \al{
    \pr{\hat K_n < \hat D_n + \delta_n} &=\sum_{i=2}^n \pr{\hat K_n < \hat D_n + \delta_n, M(1) = i}\\
    &= \sum_{i=2}^n \pr{\mF_n(\mX_i) \in \cB(\mF_n(\mX_1), \hat D_n + \delta_n), M(1) = i}\\
    &= \pr{\mF_n(\mX_2) \in \cB(\mF_n(\mX_1), \hat D_n + \delta_n)}.
    }
    Next, by the DKW inequality \citep{dvoretzky1956asymptotic} and the triangle inequality,
    \al{
    \pr{\mF_n(\mX_2) \in \cB(\mF_n(\mX_1), \hat D_n + \delta_n)} &= \pr{\mF_n(\mX_2) \in \cB(\mF_n(\mX_1), \hat D_n + \delta_n), \|\mF_n - \mF\|_\infty \leq \delta_n} + o(1)\\
    &\leq \pr{\mF(\mX_2) \in \cB(\mF(\mX_1), \hat D_n + 3\delta_n), \|\mF_n - \mF\|_\infty \leq \delta_n} + o(1).
    }
    By the continuity of $\mF(\mX)$, we have 
    $$
    \pr{\mF(\mX_2) \in \cB(\mF(\mX_1), \hat D_n + 3\delta_n) \cdl \mX_1} \las 0.
    $$
    Thus, the dominated convergence theorem tells us that $\pr{\hat K_n < \hat D_n + \delta_n} \rightarrow 0$ as well.
\end{proof}
\begin{lemma}\label{an_N_N_equiv}
As $n \rightarrow\infty$, $\pr{N(i) = \check N(i)} \rightarrow 1$.
\end{lemma}
\begin{proof}
First, choose $0 < \epsilon < \frac{1}{2} - \frac{1}{d}$ and let $\delta_n := n^{-\frac{1}{2} + \epsilon}$. Using the fact that $\|\mF_n - \mF\|_\infty = O(n^{-1/2})$, 
\al{
\pr{N(1) = \check N(1)} &=\pr{N(1) = \check N(1), \|\mF_n - \mF\|_\infty \leq \delta_n} + o(1).
}
Then, Lemma \ref{an_kn_dn} says
$$
\pr{\hat K_n < \hat D_n + 5\delta_n} \rightarrow 0.
$$
Consequently, 
\al{
\pr{N(1) = \check N(1)} = \pr{N(1) = \check N(1), \|\mF_n - \mF\|_\infty \leq \delta_n, \hat K_n \geq \hat D_n + 5 \delta_n} +  o(1).
}
Without loss of generality, assume that indices $2$ and $3$ are the nearest and second-nearest neighbors, respectively, to index $1$. The conditions that $\|\mF_n - \mF\| \leq \delta_n$ and $\hat K_n \geq \hat D_n + 5\delta_n$ imply, using the triangle inequality, that $\|\mF(\mX_1) - \mF(\mX_3)\| \geq \hat K_n - 2\delta_n \geq \hat D_n + 3\delta_n$. On the other hand, $\|\mF_(\mX_1) - \mF(\mX_2)\| \leq \hat D_n + 2\delta_n$.  Thus, $N(1) = \check N(1) = 2$.

Put together, we find that $\pr{N(1) = \check N(1)} \rightarrow 1$, concluding the proof.
\end{proof}

Using this result, we may obtain the rate of convergence of the ratio $\hat D_n/D_n$.
\begin{lemma}\label{an_dn_equiv}
    Letting $\gamma := \frac{1}{2} - \frac{1}{d}$, $n^{ \gamma}\pa{\frac{\hat D_n}{ D_n} -1} = O_{\P}(1).$
\end{lemma}
\begin{proof}
    Using Lemma \ref{an_N_N_equiv},
    \al{
    \pr{\ap{n^{\gamma}\pa{\frac{\hat D_n}{D_n} - 1}} > \epsilon} &= \pr{\ap{n^{\gamma}\pa{\frac{\hat D_n}{D_n} - 1}} > \epsilon, N(i) = \check N(i)} + o(1)\\
    &= \pr{\ap{n^{\gamma}\pa{\frac{\nm{\mF_n(\mX_1) - \mF_n(\mX_{\check N(1)})}}{D_n} - 1}} > \epsilon, N(i) = \check N(i)} + o(1).
    }
    By the triangle inequality, and the fact that $\|\mF_n - \mF\|_\infty = O_{\P}(n^{-1/2})$, 
    $$
    n^{\gamma}\pa{\frac{\nm{\mF_n(\mX_1) - \mF_n(\mX_{N(1)})}}{D_n} - 1} \leq 2n^{\gamma}\frac{\nm{\mF_n - \mF}_\infty}{D_n},
    $$
    which establishes the claim.
\end{proof}

We establish an analogous equivalence result to Lemma \ref{an_depequiv} for the nearest-neighbor directions
$$
\hat \mU_n := \frac{\mF_n(\mX_1) - \mF_n(\mX_{N(1)})}{\hat D_n}~~~{\rm and}~~~
\mU_n := \frac{\mF(\mX_1) - \mF(\mX_{\check N(1)})}{D_n}.
$$

\begin{lemma}\label{an_un_equiv}
    As $n\rightarrow \infty$, $\|\hat \mU_n - \mU_n\| \overset{\P}{\rightarrow} 0$.
\end{lemma}
\begin{proof}
Pick any $\epsilon > 0$. Using Lemma \ref{an_N_N_equiv},
$$
\pr{\nm{\hat \mU_n - \mU_n} > \epsilon} = \pr{\nm{\hat \mU_n - \mU_n} > \epsilon, N(1) =\check N(1)} + o(1).
$$
If $N(1) = \check N(1)$, using the triangle inequality,
    \al{
    \|\hat \mU_n - \mU_n\| &= \nm{\frac{\mF_n(\mX_1) - \mF(\mX_1) + \mF( \mX_{N(1)}) - \mF_n(\mX_{N(1)})}{\hat D_n}} + \nm{\frac{\mF(\mX_1) - \mF(\mX_{N(1)})}{\hat D_n} - \mU_n }.
    }
    $$
    \|\hat \mU_n - \mU_n\| \leq \frac{2 \|\mF_n - \mF\|_\infty}{\hat D_n} + \nm{\pa{\frac{D_n}{\hat D_n} - 1} \mU_n }.
    $$
    Recall that $\hat D_n = O_{\P}(n^{-1/d})$ from Corollary \ref{an_hatdn_tight}, and $\|\mF_n - \mF\|_\infty = O_{\P}(n^{-1/2})$ by \cite{dvoretzky1956asymptotic}, so that the first term goes to $0$ in probability. Finally, observing $\|\mU_n\| = 1$, and using Lemma \ref{an_dn_equiv} to show $|D_n/\hat D_n - 1| \overset{\P}{\rightarrow} 0$ concludes the proof.
\end{proof}

In view of the previous lemmas, we deduce the following weak convergence result. 
\begin{lemma}\label{an_weakcvg_xvu}
    As $n$ tends to $\infty$, $(\mX_1, n \hat D_n, \hat \mU_n) \leadsto (X_1, \mathscr{V}, \mathscr{U})$, where the conditional law of $\mathscr{V}$ given  $ \mX_1$ is that of an exponential random variable with parameter $f_X( \mF(\mX_1)) \cdot \lambda_d(\cB(0,1))$, and the conditional law of $\mathscr{U}$ given $\mathscr{V}$ and $\mX_1$ is uniform on the sphere of volume $\mathscr{V}$ centered at $\mX_1$.
\end{lemma}
\begin{proof}
    Lemma 2.1 of \cite{henze1987fraction} says that $(\mX_1, n D_n, U_n) \leadsto (\mX_1, \mathscr{V}, \mathscr{U}).$ By Lemmas \ref{an_dn_equiv} and \ref{an_un_equiv}, the random vector $(\mX_1, n\hat D_n, \hat \mU_n)$ must possess the same weak limit as that of  $(\mX_1, nD_n, U_n)$.
\end{proof}
Consequently, we obtain the main result of the section, by imitating the proof of Theorem 1.4 in \cite{henze1987fraction}.
\begin{corollary}\label{exp_avgnn} For $d > 2$,
$$
\ep{\frac{1}{n}\sum_{i,j,k=1}^n C_{i,j,k}} \rightarrow o_d~~~{\rm and}~~~
\ep{\frac{1}{n}\sum_{i,j=1}^n T_{i,j}} \rightarrow q_d
$$
as $n\rightarrow \infty$.
\end{corollary}
\subsection{Proof of asymptotic normality}
Define 
$$
\tilde Q_n = \frac{1}{n} \sum_{i=1}^n F(Y_i) \wedge F(Y_{N(i)})- \frac{1}{n(n-1)}\sum_{i=1}^n \sum_{j:j\neq i} (F(Y_i)\wedge F(Y_j)).
$$

\begin{lemma}
As $n\rightarrow \infty, \sqrt{n} (Q_n - \tilde Q_n) \overset{\P}{\rightarrow} 0.$
\end{lemma}

\begin{proof}
Writing $\sqrt{n}Q_n = \sqrt{n}\pa{Q_n^{(1)} + Q_n^{(2)}}$, for
\al{
Q^{(1)}_n &= \frac{1}{n} \sum_{i=1}^n (F_n(Y_i) \wedge F_n(Y_{N(i)})) - \frac{1}{n(n-1)}\sum_{i=1}^n \sum_{j:j\neq i} (F_n(Y_i)\wedge F_n(Y_j))\\
Q^{(2)}_n &= \frac{1}{n(n-1)}\sum_{i=1}^n \sum_{j:j\neq i} (F_n(Y_i)\wedge F_n(Y_j)) - \frac{(n+1)(2n+1)}{6n^2}, 
}
and applying Lemma D.1 of \cite{deb2020kernel} gives that $\sqrt{n} (Q_n^{(1)} - \tilde Q_n) \overset{\P}{\rightarrow} 0$.

Next, to show $\sqrt{n} Q_n^{(2)} \overset{\P}{\rightarrow} 0$, we calculate
\al{
\frac{1}{n(n-1)}\sum_{i=1}^n \sum_{j:j\neq i} (F_n(Y_i)\wedge F_n(Y_j)) &= \frac{1}{n(n-1)}\sum_{i = 1}^n (2n-2i)\cdot \frac{i}{n},
}
which implies the conclusion.
\end{proof}

The following result is a special case of Lemma 7.3 of \cite{shi2024azadkia}.  
\begin{lemma}\label{an_var_equiv}
As $n \rightarrow \infty$, $n\pa{\var{\tilde Q_n\cdl\mathcal{F}_n} - \var{\tilde Q_n}} \rightarrow 0$.
\end{lemma}

After algebraic manipulation \citep[for instance]{shi2024azadkia, han2022azadkia}, we may calculate the unconditional limiting variance in Lemma \ref{an_uncondvarcalc}.
\begin{lemma}\label{an_uncondvarcalc}
As $n\rightarrow \infty$, $n\var{\tilde Q_n} \rightarrow \sigma_d^2$, where $\sigma_d^2 = \frac{2}{5} + \frac{2}{5}q_d + \frac{4}{5}o_d$ when $d > 2$, and $\sigma_d^2 = 1$ when $d = 1$.
\end{lemma}
\begin{proof}
    By the calculations as in \cite{han2022azadkia}, we find
    \al{
    n \var{\xi_n} &= \frac{2}{5} + 2 \cdot \ep{\frac{1}{n} \sum_{i,j=1}^n T_{i,j}} + \frac{4}{5} \cdot \pa{ \ep{\frac{1}{n}\sum_{i,j,k =1}^n C_{i,j,k}} - \ep{\frac{2}{n} \sum_{i,j=1}^n T_{i,j}}} + o(1)\\
    &= \frac{2}{5} + \frac{2}{5} \cdot \ep{\frac{1}{n} \sum_{i,j=1}^n T_{i,j}} + \frac{4}{5} \cdot \pa{ \ep{\frac{1}{n}\sum_{i,j,k =1}^n C_{i,j,k}} } + o(1).
    }
    Applying Lemma \ref{an_calc_d1} and Corollary \ref{exp_avgnn} to the case of $d=1$ and $d > 2$ separately, we obtain the desired limits.
\end{proof}
From this, we derive a conditional central limit theorem.
\begin{theorem}\label{an_cond}
Conditional on $\mathcal{F}_n$, $\sqrt{n}\tilde Q_n \leadsto \mathcal{N}(0, \sigma_d^2)$ where $\sigma_d^2=1$ when $d=1$, and $ \sigma_d^2 = \frac{2}{5} + \frac{2}{5}q_d + \frac{4}{5}o_d$ when $d > 2$. 
\end{theorem}
\begin{proof}
Using a Hájek projection and the fact that $Y_i$ is independent of $\mX_i$,
$$
\frac{1}{\sqrt{n}(n-1)}\sum_{i=1}^n \sum_{j:j\neq i} (F(Y_i)\wedge F(Y_j))  = \frac{2}{\sqrt{n}}\sum_{i=1}^n h(Y_i) + o_{\P}(1)
$$
for $h(Y_i) := \ep{F(Y_i) \wedge U|Y_i} - \ep{F(Y_1) \wedge F(Y_2)}$ and $U \sim {\rm Uniform}[0,1]$ independently drawn. As a result,
$$
\sqrt{n}\tilde Q_n = \frac{1}{\sqrt{n}} \sum_{i=1}^n F(Y_i) \wedge F(Y_{N(i)})- \frac{2}{\sqrt{n}}\sum_{i=1}^n h(Y_i) + o_{\P}(1),
$$
so it suffices to show the right-hand side is asymptotically normal. 

To do so, we construct a dependency graph $\mathcal{G}_n$ consisting of nodes $\{V_i\}_{1 \leq i \leq n}$ where $V_i = (F(Y_i) \wedge F(Y_{N(i)})) - 2h(Y_i)$, and an edge between $V_i$ and $V_j$ if and only if $i = N(j)$ or $j = N(i)$ or $N(i) = N(j)$. 

Since $\ep{V_i | \mathcal{F}_n} = 0$, we apply a Berry-Esseen bound given in \cite{chen2004normal}, for any $\epsilon > 0$ to get
$$
\sup_{z }\left|\mathbb{P}\pa{\frac{\sum_{i=1}^n V_i}{\sqrt{\var{\sum_{i=1}^n V_i|\mathcal{F}_n }}}\leq t\cdl\mathcal{F}_n} - \Phi(z)\right| \leq 75C_{p+q}^{5(1 + \epsilon)}\frac{\ep{\sum_{i=1}^n |V_i|^{2 + \epsilon}|\mathcal{F}_n}}{\var{\sum_{i=1}^n V_i | \mathcal{F}_n}^{(2+\epsilon)/2}}.
$$

The numerator satisfies
\al{
n^{\epsilon/2}\sum_{i=1}^n \ep{|V_i|^{2+\epsilon}| \mathcal{F}_n} &= n^{\epsilon/2}n^{-1-\epsilon/2} \sum_{i=1}^n \ep{\pa{F(Y_i) \wedge F(Y_{N(i)})) -   2h(Y_i)}^{2+\epsilon}\cdl \mathcal{F}_n}\\
&\leq  n^{-1} \sum_{i=1}^n \ep{\pa{F(Y_i) \wedge F(Y_{N(i)})) -   2h(Y_i)}^{2+\epsilon}\cdl \mathcal{F}_n},
}
so $\sum_{i=1}^n \ep{|V_i|^{2+\epsilon}| \mathcal{F}_n}$ tends to $0$.
By the previous lemma, $\var{n^{-1}\sum_{i=1}^n V_i | \mathcal{F}_n}^{(2+\epsilon)/2}$ converges to the same limit as $\var{\tilde Q_n}$.
\end{proof}

Finally, we obtain Theorem \ref{an} by combining our previous results.

\begin{proof}[Proof of Theorem 2]
    Applying Lemma \ref{an_var_equiv}, $\frac{\var{\tilde Q_n \cdl \mathcal{F}_n}}{\var{\tilde Q_n}} \rightarrow 1.$ Therefore, it suffices to show the limiting distribution of $\frac{\sqrt{n}\tilde Q_n}{\sqrt{\var{\tilde Q_n}}}$ has a standard Gaussian limit. 
    
    Fix any $t \in \R$. Expanding, and applying Theorem \ref{an_cond} and the dominated convergence theorem yields
    $$
    \lim_{n\rightarrow \infty} \ep{\pr{\frac{\sqrt{n} \tilde Q_n}{\sqrt{\var{\tilde Q_n}}} \leq t \cdl \mathcal{F}_n} - \Phi(t)} = 0.
    $$
    Therefore, 
    $$
    \sqrt{n}\tilde Q_n \leadsto \mathcal{N}(0,\sigma_d^2),
    $$
    where $\sigma_d^2$ is given in Theorem \ref{an_cond}, concluding the proof.
    
\end{proof}

{
\bibliographystyle{apalike}
\bibliography{AMS}

\begin{thebibliography}{}

\bibitem[Ansari and Fuchs, 2022]{ansari2022simple}
Ansari, J. and Fuchs, S. (2022).
\newblock A simple extension of {A}zadkia and {C}hatterjee's rank correlation
  to a vector of endogenous variables.
\newblock {\em arXiv preprint arXiv:2212.01621}.

\bibitem[Auddy et~al., 2024]{auddy2021exact}
Auddy, A., Deb, N., and Nandy, S. (2024).
\newblock Exact detection thresholds for {C}hatterjee's correlation.
\newblock {\em Bernoulli}, 30(2):1640--1668.

\bibitem[Azadkia and Chatterjee, 2021]{azadkia2019simple}
Azadkia, M. and Chatterjee, S. (2021).
\newblock A simple measure of conditional dependence.
\newblock {\em The Annals of Statistics}, 49(6):3070--3102.

\bibitem[Azadkia et~al., 2021]{azadkia2021fast}
Azadkia, M., Taeb, A., and B{\"u}hlmann, P. (2021).
\newblock A fast non-parametric approach for causal structure learning in
  polytrees.
\newblock {\em arXiv preprint arXiv:2111.14969}.

\bibitem[Bickel, 2022]{bickel2022measures}
Bickel, P.~J. (2022).
\newblock Measures of independence and functional dependence.
\newblock {\em arXiv preprint arXiv:2206.13663}.

\bibitem[B{\"u}cher and Dette, 2024]{bucher2024lack}
B{\"u}cher, A. and Dette, H. (2024).
\newblock On the lack of weak continuity of {C}hatterjee's correlation
  coefficient.
\newblock {\em arXiv preprint arXiv:2410.11418}.

\bibitem[Cao and Bickel, 2020]{cao2020correlations}
Cao, S. and Bickel, P.~J. (2020).
\newblock Correlations with tailored extremal properties.
\newblock {\em arXiv preprint arXiv:2008.10177}.

\bibitem[Cattaneo et~al., 2025]{cattaneo2023rosenbaum}
Cattaneo, M.~D., Han, F., and Lin, Z. (2025).
\newblock On {R}osenbaum's rank-based matching estimator.
\newblock {\em Biometrika (in press)}.

\bibitem[Chatterjee, 2021]{chatterjee2021new}
Chatterjee, S. (2021).
\newblock A new coefficient of correlation.
\newblock {\em Journal of the American Statistical Association},
  116(536):2009--2022.

\bibitem[Chatterjee, 2023]{chatterjee2022survey}
Chatterjee, S. (2023).
\newblock A survey of some recent developments in measures of association.
\newblock {\em Probability and Stochastic Processes - A Volume in Honour of
  Rajeeva L. Karandikar}.

\bibitem[Chen and Shao, 2004]{chen2004normal}
Chen, L.~H. and Shao, Q.-M. (2004).
\newblock Normal approximation under local dependence.
\newblock {\em The Annals of Probability}, 32(3A):1985--2028.

\bibitem[Deb et~al., 2020]{deb2020kernel}
Deb, N., Ghosal, P., and Sen, B. (2020).
\newblock Measuring association on topological spaces using kernels and
  geometric graphs.
\newblock {\em arXiv preprint arXiv:2010.01768}.

\bibitem[Dette and Kroll, 2024]{dette2024simple}
Dette, H. and Kroll, M. (2024).
\newblock A simple bootstrap for {C}hatterjee's rank correlation.
\newblock {\em Biometrika (in press)}, page asae045.

\bibitem[Dette et~al., 2013]{dette2013copula}
Dette, H., Siburg, K.~F., and Stoimenov, P.~A. (2013).
\newblock A copula-based non-parametric measure of regression dependence.
\newblock {\em Scandinavian Journal of Statistics}, 40(1):21--41.

\bibitem[Dvoretzky et~al., 1956]{dvoretzky1956asymptotic}
Dvoretzky, A., Kiefer, J., and Wolfowitz, J. (1956).
\newblock Asymptotic minimax character of the sample distribution function and
  of the classical multinomial estimator.
\newblock {\em The Annals of Mathematical Statistics}, 27(3):642--669.

\bibitem[Fuchs, 2024]{fuchs2021bivariate}
Fuchs, S. (2024).
\newblock Quantifying directed dependence via dimension reduction.
\newblock {\em Journal of Multivariate Analysis (in press)}, 201:105266.

\bibitem[Gamboa et~al., 2022]{gamboa2020global}
Gamboa, F., Gremaud, P., Klein, T., and Lagnoux, A. (2022).
\newblock Global sensitivity analysis: a new generation of mighty estimators
  based on rank statistics.
\newblock {\em Bernoulli}, 28(4):2345--2374.

\bibitem[Griessenberger et~al., 2022]{griessenberger2022multivariate}
Griessenberger, F., Junker, R.~R., and Trutschnig, W. (2022).
\newblock On a multivariate copula-based dependence measure and its estimation.
\newblock {\em Electronic Journal of Statistics}, 16(1):2206--2251.

\bibitem[Han and Huang, 2024]{han2022azadkia}
Han, F. and Huang, Z. (2024+).
\newblock Azadkia-{C}hatterjee's correlation coefficient adapts to manifold
  data.
\newblock {\em The Annals of Applied Probability (in press)}.

\bibitem[Henze, 1987]{henze1987fraction}
Henze, N. (1987).
\newblock On the fraction of random points by specified nearest-neighbour
  interrelations and degree of attraction.
\newblock {\em Advances in Applied Probability}, 19(4):873--895.

\bibitem[Hodges and Lehmann, 1956]{MR79383}
Hodges, Jr., J.~L. and Lehmann, E.~L. (1956).
\newblock The efficiency of some nonparametric competitors of the {$t$}-test.
\newblock {\em The Annals of Mathematical Statistics}, 27(2):324--335.

\bibitem[Huang et~al., 2022]{huang2020kernel}
Huang, Z., Deb, N., and Sen, B. (2022).
\newblock Kernel partial correlation coefficient---a measure of conditional
  dependence.
\newblock {\em The Journal of Machine Learning Research}, 23(1):9699--9756.

\bibitem[Kroll, 2024]{kroll2024asymptotic}
Kroll, M. (2024).
\newblock Asymptotic normality of {C}hatterjee's rank correlation.
\newblock {\em arXiv preprint arXiv:2408.11547}.

\bibitem[Lin and Han, 2022]{lin2022limit}
Lin, Z. and Han, F. (2022).
\newblock Limit theorems of {C}hatterjee's rank correlation.
\newblock {\em arXiv preprint arXiv:2204.08031}.

\bibitem[Lin and Han, 2023]{lin2021boosting}
Lin, Z. and Han, F. (2023).
\newblock On boosting the power of {C}hatterjee's rank correlation.
\newblock {\em Biometrika}, 110(2):283--299.

\bibitem[Lin and Han, 2025]{lin2024failure}
Lin, Z. and Han, F. (2025).
\newblock On the failure of the bootstrap for {C}hatterjee's rank correlation.
\newblock {\em Biometrika (in press)}.

\bibitem[R{\'e}nyi, 1959]{renyi1959measures}
R{\'e}nyi, A. (1959).
\newblock On measures of dependence.
\newblock {\em Acta Mathematica Hungarica}, 10(3-4):441--451.

\bibitem[Rosenbaum, 2005]{rosenbaum2005exact}
Rosenbaum, P.~R. (2005).
\newblock An exact distribution-free test comparing two multivariate
  distributions based on adjacency.
\newblock {\em Journal of the Royal Statistical Society Series B: Statistical
  Methodology}, 67(4):515--530.

\bibitem[Rosenbaum, 2010]{rosenbaum2010design}
Rosenbaum, P.~R. (2010).
\newblock {\em Design of Observational Studies}.
\newblock Springer.

\bibitem[Schweizer and Wolff, 1981]{schweizer1981nonparametric}
Schweizer, B. and Wolff, E.~F. (1981).
\newblock On nonparametric measures of dependence for random variables.
\newblock {\em The Annals of Statistics}, 9(4):879--885.

\bibitem[Shi et~al., 2022]{shi2020power}
Shi, H., Drton, M., and Han, F. (2022).
\newblock {On the power of Chatterjee's rank correlation}.
\newblock {\em Biometrika}, 109(2):317--333.

\bibitem[Shi et~al., 2024]{shi2024azadkia}
Shi, H., Drton, M., and Han, F. (2024).
\newblock On {A}zadkia--{C}hatterjee's conditional dependence coefficient.
\newblock {\em Bernoulli}, 30(2):851--877.

\bibitem[Strothmann et~al., 2024]{strothmann2024rearranged}
Strothmann, C., Dette, H., and Siburg, K.~F. (2024).
\newblock Rearranged dependence measures.
\newblock {\em Bernoulli}, 30(2):1055--1078.

\bibitem[Zhang, 2023a]{zhang2023relationships}
Zhang, Q. (2023a).
\newblock On relationships between {C}hatterjee's and {S}pearman's correlation
  coefficients.
\newblock {\em arXiv preprint arXiv:2302.10131}.

\bibitem[Zhang, 2023b]{zhang2022asymptotic}
Zhang, Q. (2023b).
\newblock On the asymptotic null distribution of the symmetrized {C}hatterjee's
  correlation coefficient.
\newblock {\em Statistics and Probability Letters}, 194:109759.

\end{thebibliography}
}
\end{document}